\documentclass[12pt]{amsart}

\usepackage{amssymb}

\usepackage{amsfonts,amssymb,amsmath}
\usepackage{amsthm, eucal}
\usepackage{graphics}
\usepackage[all]{xy}
\usepackage[dvips]{graphicx}
\usepackage{mathrsfs} 

\makeatletter
\def\blfootnote{\xdef\@thefnmark{}\@footnotetext}
\makeatother

\let\oldmarginpar\marginpar
\renewcommand\marginpar[1]{\-\oldmarginpar[\raggedleft\footnotesize #1]%
{\raggedright\footnotesize #1}}

\usepackage{color}

\newtheorem{theorem}{Theorem}[section]
\newtheorem{lemma}[theorem]{Lemma}
\newtheorem{prop}[theorem]{Proposition}
\newtheorem{fact}[theorem]{Fact}
\newtheorem{cor}[theorem]{Corollary}

\theoremstyle{remark}

\theoremstyle{definition}
\newtheorem{example}[theorem]{Examples}

\numberwithin{equation}{section}

\let\leq=\leqslant
\let\geq=\geqslant

\newcommand{\mf}[1]{\mathfrak{#1}}
\newcommand{\germ}[1]{\mathfrak{#1}}
\newcommand{\mb}[1]{\mathbb{#1}}

\newcommand{\ord}{\mathrm{ord}}
\newcommand{\eps}{\varepsilon}
\newcommand{\bos}{\mathbf{s}}
\newcommand{\tbos}{\tilde{\bos}}
\newcommand{\bft}{\mathbf{t}}

\newcommand{\Z}{\mathbb{Z}}
\newcommand{\F}{\mathbb{F}}
\newcommand{\bF}{\bar{\F}}
\newcommand{\N}{\mathbb{N}}

\newcommand{\rst}{\mathrm{Res}}

\newcommand{\Ker}{\mathrm{Ker}}
\newcommand{\Zen}{\mathrm{Z}}
\newcommand{\argu}{\hbox to 7truept{\hrulefill}}
\newcommand{\hbos}{\hat{\bos}}

\newcommand{\rk}{\mathrm{rk}}
\newcommand{\triv}{\{1\}}
\newcommand{\idn}{\mathrm{Ind}}
\newcommand{\tG}{\Phi(G,\Z_p)}
\newcommand{\eG}{\Phi(G,\F_p)}
\newcommand{\hGt}{\Upsilon(G,\sigma,2)}

\newcommand{\image}{\mathrm{Im}}
\newcommand{\sat}{\mathrm{sat}}
\newcommand{\Ext}{\mathrm{Ext}}

\newcommand{\iid}{\mathrm{id}}

\newcommand{\Syl}{\mathrm{Syl}}
\newcommand{\Odd}{\mathrm{Odd}}

\newcommand{\Sl}{\mathrm{SL}}

\newcommand{\bG}{\bar{G}}

\newcommand{\End}{\mathrm{End}}
\newcommand{\Der}{\mathrm{Der}}

\newcommand{\Q}{\mathbb{Q}}

\emergencystretch=\maxdimen
\numberwithin{equation}{section}

\begin{document}	
	\title[\(p\)-adic analytic profinite CA-groups]{Compact \(p\)-adic analytic groups in which centralizers are abelian}
	\author[L.~Mendon\c ca]{Luis Mendon\c ca}
	\author[Th.~Weigel]{Thomas S. Weigel}
	\author[T.~Zapata]{Theo Zapata}
	\address{Departamento de Matem\'atica, Universidade Federal de Minas Gerais, Belo Horizonte, Brazil}
	\email{luismendonca@mat.ufmg.br}
	\address{Department of Mathematics and Applications, University of Milano Bicocca,  Milan (Italy)}
	\email{thomas.weigel@unimib.it}
	\address{Departamento de Matem\'atica, Instituto de Ci\^encias Exatas, Universidade de Bras\'ilia, Brasilia, Brazil}
	\email{zapata@mat.unb.br}
	\date{\today}

\begin{abstract} 
Using methods of associative algebras, Lie theory, group cohomology, and modular representation theory, we construct profinite $p$-adic analytic groups 
such that the centralizer of each of their non-trivial elements is abelian.
The paper answers questions of P.~Shumyatsky, P.~Zalesskii, and T.~Zapata in the Israel J. Math., v.~230, 2019.
\end{abstract}
	
	
	

	\maketitle


\section{Introduction}

Describing centralizers is often one of the fundamental strides in understanding a group.
The condition that the centralizer of every non-trivial element in a group is abelian defines a CA-group and, in recent years, the equivalent condition of the transitivity of commutation among non-trivial elements, has attracted much attention in Group Theory and other areas of Algebra (see \cite{Wu}, \cite{FR:08} and \cite{ShZZ:19}). 

The classification of finite (and locally finite) CA-groups has been the subject of considerable study, spanning from 1925 through to 1998. Weisner initially showed that finite CA-groups are either simple or solvable (cf.~\cite{Weisner:25} and \cite{Wu}). Then in the Brauer--Suzuki--Wall  theorem (see \cite{BSW:58}) finite CA-groups of even order fall into a trichotomy: they are either abelian, Frobenius groups, or two-dimensional projective special linear groups over a finite field of even order, namely $\textup{PSL}(2,2^m)$ for $m \geq 2$. Of particular note is the 1957 landmark contribution of M.~Suzuki \cite{Suzuki:57}, wherein he proved that simple CA-groups must necessarily be of even order. Suzuki’s work was a crucial precursor to the outstanding Feit--Thompson theorem on the solubility of groups of odd order.

Moreover, the significance of infinite CA-groups in contemporary mathematics is underscored by their connection with the theory of residually free groups (cf.~\cite[p.~10]{LyndonSchupp}, \cite{Wu}, \cite{FR:08} and \cite{AFW:15}). Specifically, finitely generated residually free CA-groups are limit groups, which played a pivotal role in the solution of the Tarski problems. In these groups, each centralizer is also malnormal (i.e., any conjugate by an element outside of the subgroup intersects trivially the subgroup itself); such groups are termed \textit{CSA-groups}. Examples of infinite CA-groups are plentiful and include free groups, torsion-free hyperbolic groups, certain subgroups of $\textup{PSL}(2,\mathbb{C})$, free Burnside groups of large odd exponent, free soluble groups, Tarski monsters, one-relator groups with torsion, certain small cancellation groups, and the fundamental groups of certain compact 3-manifolds.

Profinite CSA-groups, however, are still quite mysterious. 
In response to \cite[Question~3]{ShZZ:19}, we offer the following class of profinite CSA-groups, thereby answering the posed question affirmatively.

\begin{theorem} \label{t:CSA}
    Let $D$ be a central division algebra of prime degree $\ell$ over $\mathbb{Q}_p$.
    If $p$ is odd and $\ell$ does not divide $p-1$,
    then the profinite $p$-adic analytic group $\Sl_1(D)$ is a CSA-group which is neither pro-$p$ nor abelian.
\end{theorem}

To our knowledge the first paper which studied infinite profinite CA-groups is \cite{ShZZ:19}. Its main result states that such profinite groups possess 
an open subgroup which is either pro-$p$ or abelian.

In the present paper 
we focus our attention on compact $p$-adic analytic groups, a class of profinite groups extensively studied
by Lazard in \cite{laz:blue}.
 From the structure theory of compact $p$-adic analytic  groups and from the Lie correspondence between such groups and their $p$-adic Lie algebras
 (see \cite[Ch.~8-9]{DDMS:padic}),
 one deduces that in order to describe the compact $p$-adic analytic CA groups one has to accomplish a classification of finite dimensional $p$-adic Lie algebras in which the centralizer of each of its nonzero elements is abelian.

Now let $L$ be $k$-Lie algebra acting on a $k$-Lie algebra $I$ by derivations, i.e., one has a homomorphism of $k$-Lie algebras $\alpha\colon L\to\Der_k(I)$. We denote the associated semi-direct product by $L\ltimes_\alpha I$.
One says that  $L$ is acting
\emph{fixed-point freely} on $I$, if 
for $x\in L$ and $y\in I$, $\alpha(y)(x)=0$ implies
$x=0$ or $y=0$.

\begin{theorem}\label{t:CTLieAlgebras}
Let $k$ be a finite extension of $\mathbb{Q}_p$. Then a finite dimensional Lie algebra over $k$ is CA if, and only if, it is isomorphic to either: \\
\textup{(a)} $k^n$ or $k^m\ltimes_\alpha k^n$, $m\leq n$,
where $\alpha\colon k^m\to\End_k(k^n)$ is a fixed-point free action; \\
\textup{(b)} the scalar restriction $R_{K|k}(L_0)$, for a finite extension $K|k$, where $L_0$ is isomorphic to
\textup{(b1)} a split form of $\mathfrak{sl}_2$, or \textup{(b2)} the derived Lie subalgebra of a central division $K$-algebra of prime degree; \\
\textup{(c)} $ R_{K|k}(L_0)\ltimes_\beta k^n$, for $L_0$ of type \textup{(b2)}, where $\beta\colon R_{K|k}(L_0)\to\End_k(k^n)$ is a fixed-point free action.
\end{theorem}

\begin{cor}  \label{quot}
 Let $L$ be a Lie algebra of finite dimension over a $p$-adic field $k$. If $L$ is CA, then so are all of its quotients.
\end{cor}

\begin{cor} \label{GpQuot}
If a finite dimensional $p$-adic Lie group is virtually CA, then so is each of its quotients.
\end{cor} 

Examples of CA finite dimensional $p$-adic Lie groups which are virtually abelian possessing quotients that are not CA were constructed in \cite[Sec.~4]{ShZZ:19}. A unified and more conceptual construction enveloping all such examples is provided by the corollaries of Theorem~\ref{thm:exA}.

An infinite profinite group is said to be \emph{$p$-Bieberbach} if it contains
an open normal subgroup which is non-trivial torsion-free finitely generated pro-$p$ abelian.

For every finite group $G$ and every prime number $p$ there exists a universal abelian
pro-$p$ Frattini extension
\begin{equation}
\label{eq:UFratt}
\tbos_G:\qquad\xymatrix{0\ar[r]&\Omega_2(G,\Z_p)\ar[r]^-{\iota}&\tG\ar[r]^-{\pi_G}&G\ar[r]&\triv},
\end{equation}
which has been introduced by W.~Gasch\"utz and studied in detail by K.~Gruenberg (cf. \cite[Ch.~9-11]{KG:coh} and \cite[Ch.~22]{FJ:fiar}). In particular, $\tG$ is a $p$-Bieberbach group whenever $p$ is dividing the order of $|G|$ (see \S\ref{ss:propFratt}, Fact~\ref{fact:fratspl}). 
If $\sigma$ is a central involution
in $G$ there exists a universal pro-$2$ extension
\begin{equation}
\label{eq:UmO}
\hat{\bos}_G:\qquad\xymatrix{0\ar[r]&\Omega_2(G,\Z_2)^-\ar[r]^-{j}&\Upsilon(G,\sigma,2)\ar[r]^-{\tau_G}&G\ar[r]&\triv,}
\end{equation}
where $\Omega_2(G,\Z_2)^-$ is the universal quotient of $\Omega_2(G,\Z_2)$ on which $\sigma$ is acting as $-1$. 
For simplicity we henceforth write
\[ \Phi=\Phi(G,\mathbb{Z}_p), \Omega=\Omega_2(G,\mathbb{Z}_p), \Upsilon=\Upsilon(G,\sigma,2), \text{ and } \Omega^-=\Omega_2(G,\Z_2)^-\, , \]
whenever no ambiguity arises.
The following fact is a direct consequence of Proposition~\ref{prop:2grp}.

\begin{fact}
For $p=2$ the group $\Upsilon$ is a $2$-Bieberbach group unless $G\simeq P\ltimes O_{2^\prime}(G)$ for some
finite cyclic $2$-group $P$ in which case one has $\Upsilon \simeq G$.
\end{fact}

For a finite group $G$ we denote by $\Odd(G)$ the set of its subgroups of odd order.
For short, a finite group $G$ together with a central involution $\sigma$ in $G$ will be said to be \emph{$\Sl_2$-like}, if:
\begin{itemize}
\item[(i)] every subgroup $H$ in $\Odd(G)$ is abelian;
\item[(ii)] $\sigma$ is the unique element of order $2$ in $G$;
\item[(iii)] Sylow $2$-subgroups of $G$ are non cyclic.
\end{itemize}

One has the following list of examples of $\Sl_2$-like groups.
\begin{example}
\label{ex:Sl2}
\noindent
{\rm (a)} an $\Sl_2$-like finite $2$-group must be a generalized quaternion group $Q_{2^n}$ of order $2^n$ 
for some positive integer $n\geq 3$.

\noindent
{\rm (b)} the groups $G=\Sl_2(q)$, $q=p^f$, $p$ odd, are $\Sl_2$-like. Note that this class of groups includes: $\Sl_2(3)\simeq C_3\ltimes Q_8$;
$\Sl_2(5)=2.A_5$, which coincides with the
double cover of the alternating group $A_5$ of degree $5$;
$\Sl_2(9)=2.A_6$, which coincides with the
double cover of the alternating group $A_6$ of degree $6$.
\end{example}

The examples (b) have been the motivation for us to call groups with the properties (i) -- (iii) to be $\Sl_2$-like.

A main goal of this paper is to study the $2$-Bieberbach groups $\Upsilon$ for an $\Sl_2$-like finite group $G$.
For a finite group $(G,\sigma)$ with central involution $\sigma\in\Zen(G)$ and non-cyclic Sylow $2$-group $P\in\Syl_2(G)$ the $2$-Bieberbach group $\Upsilon$ is a non-trivial $2$-Frattini extension of $G$, i.e., there exists a surjective continuous group homomorphism $\tau_G\colon\Upsilon\to G$ satisfying
$\operatorname{Ker}(\tau_G)\subseteq\Phi(\Upsilon)\cap O_2(\Upsilon)$, where $\Phi(\argu)$ denotes the \emph{Frattini subgroup} of a
profinite group, and $O_2(\argu)$ denotes the maximum pro-$2$ closed normal subgroup of a profinite group.

The following theorem shows that for a $\Sl_2$-like finite group $G$ half of the centralizers of the $2$-Bieberbach group $\Upsilon$  are almost abelian.

\begin{theorem}
\label{thm:exA}
Let $G$ be an $\Sl_2$-like finite group. 
Let $x\in\Upsilon$ with $\tau_G(x)$ of even order.
Then $C_{\Upsilon}(x)$ is finite and $C_{\Upsilon}(x)\simeq P\ltimes O_{2^\prime}(C_\Upsilon(x))$ for some finite cyclic $2$-group $P$.
\end{theorem}

From Theorem~\ref{thm:exA} and item (a) in Examples \ref{ex:Sl2} one obtains the following result, thus answering \cite[Question 4]{ShZZ:19} affirmatively.

\begin{cor}
\label{cor:exB}
For each $n\geq 3$, the pro-$2$\ group\  $\Upsilon(Q_{2^n},\sigma,2)$  is CA.
\end{cor}

It should be mentioned that for most examples of $\Sl_2$-like finite groups $G$, the $2$-Bieberbach group $\Upsilon(G,\sigma,2)$
will not be CA, i.e., for $G=2.A_5$, the group $\Upsilon(G,\sigma,2)$
contains involutions whose centralisers are isomorphic to the finite groups 
$C_4\ltimes C_3$ and $C_4\ltimes C_5$, respectively, where $C_n=\Z/n\Z$ denotes the cyclic group of order $n$, and the involution $\sigma\in C_4$ acts trivially on $C_3$ or $C_5$, respectively.
Nevertheless, this approach provides also the following example
(cf. Corollary~\ref{cor:New} and \S\ref{ss:2mod}).

\begin{cor}
\label{cor:new}
The $2$-Bieberbach group $\Upsilon(\Sl_2(3),\sigma,2)$  is CA.
\end{cor}

While Theorem~\ref{thm:exA} describes the centralisers of elements
$g\in\Upsilon(G,\sigma,2)$ for which $\tau_G(g)$ is of even order,
a certain fixed-point property for subgroups $H \in \Odd(G)$  will turn out to be important in order to describe the centralisers of elements $g\in\Upsilon(G,\sigma,2)$
with $\tau_G(g)$ of odd order (see Subsection~\ref{ss:FixedPointProperty}).

\section{Lie algebras}

Throughout this section let $k$ be a finite extension of $\mb{Q}_p$, for a fixed prime $p$, and let $L$ be a finite-dimensional Lie algebra over $k$. For convenience, we state some basic results. 

\begin{lemma}[\protect{\cite[Prop.~2.1]{KlMo10}}]\label{l:indecomposable}
If $L$ is CA, then so are all of its subalgebras.
If, moreover, $L$ is nonabelian, then it is directly indecomposable, i.e., $L=A \times B$, for subalgebras $A$ and $B$, implies $A=0$ or $B=0$.
\end{lemma}

\begin{lemma}[\protect{\cite[Prop.~2.2]{KlMo10}}] \label{l:f-p-fExt}
Let $A$ be an abelian ideal of $L$.
If $L/A$ is CA and acts fixed-point-freely by inner derivations on $A$,
then $L$ is CA.
\end{lemma}

\subsection{Solvable Lie algebras}
\begin{prop}
 If $L$ is solvable and CA, then $L$ is either abelian, or metabelian of the form $C \ltimes N$, where $N$ and $C$ are abelian with $\dim C \leq \dim N$ and $C$ acts fixed-point-freely on $N$.
\end{prop}

\begin{proof}
Apart from the inequality on dimensions all can be found in \cite[Th.~2.3]{KlMo10}. Now, since $C$ acts fixed-point-freely on $N$, if $y_1, \ldots, y_n \in C$ are linearly independent, then for any $x$ in $N \smallsetminus \{0\}$ the elements
 $[x,y_1], \ldots, [x,y_n] \in N$ are also linearly independent. Indeed, from $\sum_{i=1}^n a_i [x,y_i]=0$ with $a_i \in k$, we get that $\sum_{i=1}^n a_i y_i$ fixes $x$, 
 which by fixed-point-freeness gives us $a_i=0$ for all $i$. This proves the inequality.
\end{proof}

Notice that any possible dimensions in the proposition above do occur: given $0< r \leq n$, find a field extension $K$ of $k$ of degree $n$, and consider the Lie $K$-algebra with basis $\{x,y\}$ and bracket $[x,y]=y$. It is CA, and as a Lie $k$-algebra it is of the form $k^n \ltimes k^n$. Clearly it contains subalgebras of the form $k^r \ltimes k^n$, and these must be CA too.

\subsection{Simple Lie algebras}
Recall first that the \textit{$k$-rank} of a semisimple Lie algebra over a field $k$ of characteristic zero is the dimension of any of its maximal split toral subalgebras.
Recall also that a Lie algebra is called \textit{absolutely simple} if its scalar extension to each field extension is simple. 
By passing to the centroid $K$ of $L$, the next lemma leaves us with the investigation of the CA property for absolutely simple Lie $K$-algebras of $K$-rank at most one.

\begin{lemma}
Let $L$ be a nonabelian reductive Lie algebra over $k$. If $L$ is CA, then it is simple of $k$-rank at most $1$.
\end{lemma}

\begin{proof}
By Lemma~\ref{l:indecomposable}, $L$ must be simple. 
Let $S$ be a maximal split toral subalgebra of $L$ and $\alpha$ be any root of $S$ in $L$. 
By the Jacobson-Morozov theorem and the weight decomposition (cf.~\cite[Prop.~3.1.9]{Sch17}) there is $h_{\alpha}$ in $S$ such that for any nonzero $x_{\alpha}$ in the weight space of $\alpha$ one has $[h_{\alpha},x_{\alpha}]=2 x_{\alpha}$. 
If $\dim S$  were $\geq 2$, there would be a nonzero element $s$ in $\Ker(\alpha)$. 
Then
 $[h_{\alpha},s]=0=[s,x_{\alpha}]$, but $[h_{\alpha},x_{\alpha}]\neq 0$, which contradicts the CA condition.
\end{proof}

Recall that the \emph{anisotropic kernel} $A(L)$ of a semisimple Lie algebra $L$ is the derived subalgebra of the centralizer $Z_L(S)$ of a maximal split toral subalgebra $S$ of $L$. It is a semisimple Lie subalgebra of $L$ and, as the name suggests, it is itself an anisotropic Lie algebra. When $A(L)=0$, we say that $L$ is \emph{quasi-split}.

\begin{prop}
Let $L$ be an absolutely simple Lie algebra over $K$ of $K$-rank one. Then $L$ is commutative transitive if and only if 
$L$ is a form of $\mf{sl}_2$.
\end{prop}

\begin{proof}
Let $S$ be a maximal split toral subalgebra of $L$. 
The CA condition imposes $A(L)=0$, that is, $L$ is quasi-split.
If $L$ is actually split, then clearly $L$ is a form of $\mf{sl}_2$. 
Now, according to \cite[Sec.~4.6]{Sch17}, if $L$ is non-split quasi-split of rank $1$, then it can only be a form of $\mf{sl}_3$ with the following realization for some choice of quadratic extension (see also \cite[Ex.~3.2.10]{Sch17}).

For $F|K$ a quadratic extension, choose $y \in F \smallsetminus K$ with $y^2 \in K$. We consider the algebra of matrices of the form
\[ \begin{pmatrix}
    a+by & c+dy & ey \\
    f+gy & -2by & -c+dy \\
    hy & -f+gy & -a+by
   \end{pmatrix}
\]
with $a,b,c,d,e,f,g,h \in K$, endowed with the usual operations. 
This algebra is not CA, which may be verified by looking at the following triple of elements:
\[ \begin{pmatrix}
    y & 0 & 0 \\
    0 & -2y & 0 \\
    0 & 0 & y
   \end{pmatrix},
   \begin{pmatrix}
    0 & 0 & y \\
    0 & 0 & 0 \\
    0 & 0 & 0
   \end{pmatrix} \hbox{ and }
   \begin{pmatrix}
    0 & y & 0\\
    0 & 0 & y \\
    0 & 0 & 0
   \end{pmatrix}.
\]

Conversely, suppose $L$ is a form of $\mf{sl}_2$. 
Evidently, if a scalar extension of a Lie algebra is CA, then the algebra itself is CA. 
Note that $\mf{sl}_2(F)$ is CA if and only if the characteristic of the field $F$ is $\neq 2$.  
\end{proof}

Finally, in the case of $K$-rank zero, that is, of anisotropic Lie algebras, we have the following.

\begin{prop}  \label{LieDivisionAlg}
Let $L$ be an absolutely simple anisotropic Lie algebra over $K$. Then $L$ is commutative transitive if and only if $L$ is the derived Lie subalgebra of a central division $K$-algebra of prime degree.
\end{prop}

\begin{proof}
From Veisfeiler's classification of semisimple $p$-adic Lie algebras (\cite[Th.~2]{Veisfeiler:64}) we have that $L$ is an inner form of type $A$. 
More precisely, by \cite[Prop.~4.5.21]{Sch17}, such $L$ of type $A_{n-1}$ is isomorphic to the derived Lie subalgebra $[D,D]$ of a central division $K$-algebra $D$ (endowed with the bracket given by $[a,b] = ab - ba$) of degree $n$.

The case $n=2$ means that a scalar extension of $L$ is $\mf{sl}_2$, so that $L$ is clearly CA.

Suppose $n$ is a prime $>2$. 
Note that the Lie-algebraic and the associative-algebraic centralizers of an element in $D$ coincide.
If $d\in D\smallsetminus K$, then the associative subalgebra generated by $d$ is a field  extension of $K$ of degree $n$ which is contained in the centralizer of $d$ in $D$, so these two subalgebras coincide.
Since $D$ is reductive (because its scalar extension to a splitting field of $D$ is $\mathfrak{gl}_n$) it follows that $[D,D]$ trivially intersects $Z(D)=K$. 
Thus $L = [D,D]$ is CA.

Next we use the fact that each  finite dimensional simple division algebra over a $p$-adic field is \emph{cyclic}
(see \cite[p.~184, Sec.~IX-4]{Weil:67}).
In particular, for some cyclic extension $W|K$ of degree $n$ and some $x $ in $D$ we have
\[ D = W \oplus W x \oplus \ldots \oplus W x^{n-1}, \]
and for $z$ in $W$ the multiplication satisfies $x z =  {\sigma}(z) x$, where $\sigma$ is a generator of the Galois group of $W|K$.
We claim that the derived subalgebra $L$ of $D$ can be written as a vector space over $K$ as
\[ L = W_0 \oplus W x \oplus \ldots \oplus W x^{n-1},\]
where $W_0$ is a complement to $K$ in $W$.
Indeed, for $z,z'\in W$ and $1 \leq i \leq n-1$, we have
\begin{equation} \tag{*} 
[z,z'x^i]= (z-\sigma^i(z))z' x^i\, .
\end{equation}
By taking $z$ not fixed by $\sigma^i$ and $z'$ any multiple of ${(z-\sigma^i(z))^{-1}}$, we see that $L$ contains $W x \oplus \ldots \oplus W x^{n-1}$. 
To establish the claim note that $\dim_K L=n^2-1$ (because $L$ is a form of $\mathfrak{sl}_n$).

Finally, suppose that $n$ is not a prime. 
Let $i>1$ be a proper divisor of $n$ and let $M$ be the fixed field of $\langle \sigma^i \rangle$. 
Pick $a \in W_0 \smallsetminus M \cap W_0$ and $b\in M \cap W_0 \smallsetminus \{0\}$. Notice that such choices are possible, since $M \cap W_0$ has codimension one in $M$. Then $a,b,bx^i\in L$, $[a,b]=0$ and, by equation (*), $[b,bx^i]=0\neq [a,bx^i]$. 
So $L$ is not CA.
\end{proof}

\subsection{Proof of Theorem \ref{t:CTLieAlgebras} and its corollaries}
\begin{proof}[Proof of Theorem \ref{t:CTLieAlgebras}] 
By Levi's theorem $L$ is the semidirect product of a subalgebra $Q$ of $L$ and the solvable radical $R$ of $L$. 
Then both $Q$ and $R$ are CA and $Q$ is semisimple. 
In virtue of Lemma~\ref{l:f-p-fExt}, to conclude the proof it remains to consider the case when both $R$ and $Q$ are nonzero.

Let $A$ be the largest nilpotent ideal of $R$ and let $Q$ act on $R$ by restriction of the adjoint action.
Since $[L,R]\subseteq A$ (\cite[Prop. I.5.2.6]{BourbakiLIE}), by complete reducibility one has $R = A \oplus B$ for some $Q$-invariant complement $B$. 
Thus, $[Q,B]\subseteq A\cap B=\{0\}$. 
If $B$ were $\neq\{0\}$, then $[Q,Q]=\{0\}$ by CA.
Therefore $R=A$. 

So $L = Q \ltimes A$, with abelian $A$, since $A$ is the largest nilpotent ideal of the Lie algebra $L$ which is CA. 
\medskip
Each nilpotent element of $Q$ acts via a nilpotent matrix of $\mf{gl}(A)$, so it is zero.
Thus $Q$ is an anisotropic $\mf{sl}_\ell(K)$-form for some prime $\ell$ and some extension $K$ of $k$. 
To see that in this case the action is indeed fixed-point-free, consider any non-trivial element $q\in Q$ and suppose that $q\cdot a=0$ for some $a \neq 0$. 
Then actually $q \cdot b=0$ for all $b \in A$. 
Moreover, for any $q' \in Q$ and $b \in A$ we have:
\[ [q,q']\cdot b = q\cdot(q' \cdot b) -q'\cdot(q \cdot b) = 0. \]
It follows that the ideal generated by $q$ commutes with $A$. 
By simplicity, $[A,Q]=0$, which is a contradiction. \qedhere
\end{proof}

\begin{proof}[Proof of Corollary \ref{quot}]
Let us analyse each case of Theorem~\ref{t:CTLieAlgebras}. 
If $L$ is abelian or simple, there is nothing to do; so let $I$ be a non-zero ideal of $L$. 

Suppose $L = C \ltimes N$ is metabelian. 
If $x\in I \smallsetminus N$, then $I$ contains $N$. Indeed, $x$ acts with an isomorphism on $N$, so every $a\in N$ is of the form $[x,a']$ for some $a' \in N$. This implies that $L/I$ is isomorphic to a subalgebra of $L/N\simeq C$, which is abelian.  
If $I \subseteq N$, then $L/I \simeq  C \ltimes (N/I)$, and it is easily seen that the action of $C$ on $N/I$ is fixed-point-free.

Finally, suppose $L=Q \ltimes R$ with abelian $R$ and simple $Q$ as in case (c). 
If $x\in I \smallsetminus R$, then the quotient $L/I$ is abelian, since $Q$ is simple. 
If $I \subseteq R$, then, since $I$ is closed under the action of $Q$, we can find a complement $J$ by complete reducibility. 
Thus $L/I \simeq Q \ltimes J$, which is clearly another algebra of type (c). \qedhere
\end{proof}

\begin{proof}[Proof of Corollary \ref{GpQuot}]
Let $N$ be a normal Lie subgroup of $G$ and let $U$ be a closed finite index subgroup of $G$ such that $U$ is uniformly powerful pro-$p$ and CA. 
Then $L(U)$ and hence $L(G)$ is a CA finite dimensional Lie algebra over $\mathbb{Q}_p$. 
Since $G\to G/N$ induces an isomorphism $L(G)/L(N)\to L(G/N)$, some uniformly powerful pro-$p$ closed finite index subgroup of $G/N$ must be CA. \qedhere
\end{proof}

\subsection{$p$-adic Lie groups with CA Lie algebras}

A $p$-adic analytic group having corresponding Lie algebra with the CA property need not be CA; e.g., the {smallest torsion-free pro-$p$ group non-CA}, $\mathbb{Z}_2\ltimes \mathbb{Z}_2$, where the action is by inversion outside of the kernel $=2\mathbb{Z}_2$; 
it is a pro-2 group of rank $2$ and dimension $2$ (since it is {virtually-$\mathbb{Z}_2^2$}) whose corresponding Lie algebra over $\mathbb{Q}_p$ is 
CA. 

Let $G$ be a $p$-adic analytic group with Lie algebra $L(G)$.
If $L(G)$ is CA, then $G$ is virtually CA.
Indeed, let $U$ be a uniformly powerful open subgroup of $G$;
from the Campbell-Baker-Hausdorff formula it follows that its corresponding $\mathbb{Z}_p$-Lie lattice $L_U$ is CA if and only if so is $U$.
Conversely, if $G$ is CA then so is $L(G)$; for, $L(G)=\mathbb{Q}_p \otimes _{\mathbb{Z}_p} L_U$ where $\mathbb{Q}_p$ is the field of fractions of $\mathbb{Z}_p$ and $L_U$ is $p$-torsion-free.
More precisely:

\begin{prop}
\label{prop:powful}

\noindent
{\rm (a)}
Let $G$ be a uniformly powerful pro-$p$ group.
Then $G$ is CA if, and only if, $L(G)$ is a CA $\Q_p$-Lie algebra.

\noindent{\rm (b)}
If $M$ is a CA $\Q_p$-Lie algebra and $\Lambda\subseteq M$ is 
$\Z_p$-sublattice of $M$, then $p\Lambda\subseteq M$ is a powerful
$\Z_p$-sublattice of $M$, and $G=(p\Lambda,\phi_h)$, where $\phi_h\colon\Lambda\times\Lambda\to\Lambda$ is given by substitution of the Baker-Campbell-Hausdorff element, is a uniformly
powerful pro-$p$ group, which is CA.
\end{prop}

The case when $L(G)$ is a nonabelian CA reductive Lie algebra is the most interesting one (Theorem~\ref{t:CTLieAlgebras}(b)). 
This is the case for any non-soluble {just-infinite} $p$-adic analytic pro-$p$ group $G$ such that $L(G)$ is CA. 
Moreover, regarding Lie's third theorem, one has that each CA simple Lie algebra over $\mathbb{Q}_p$ is isomorphic to $L(G)$ for some {just-infinite} $p$-adic analytic pro-$p$ group $G$.

Let $A$ be a finite dimensional central simple algebra over a field $K$.
By definition, $\textup{SL}_1(A)$ consists of the elements in $A^\times$ of \textit{reduced norm} $=1$ in $K^\times$ (see \cite[Ch.~IX-2]{Weil:67} for the italicized term).
The following proposition gives, for instance, that the $p$-adic analytic group $\textup{SL}_1(A)$ is CA where $A$ a central division algebra over $\mathbb{Q}_p$ of dimension $p^2$ with $p$ an odd prime. We begin with an immediate lemma.

\begin{lemma}
\label{l:units'}
Let $A$ be a ring with the property that 
the centralizer of each noncentral element in $A$ is commutative.
Then every subgroup of $A^\times$ intersecting the center of $A$ trivially is CA.
\hfill \qedsymbol
\end{lemma}

\begin{prop}
\label{p:SL1'}
Let $K$ be a field with no $n$-th root of unity other than $1$.
If $A$ is a central simple algebra over $K$ of dimension $n^2$ such that every commutative subalgebra other than $K$ is maximal, then the group $\textup{SL}_1(A)$ is CA.
\end{prop}

\begin{proof}
First, the ring $A$ has the property stated in Lemma~\ref{l:units'}.
Indeed, if $x\in A \smallsetminus Z(A)$ and $y\in C_A(x)$ then the subalgebras generated by $x$ and by $x$ and $y$ are commutative and must coincide; 
hence they are $C_A(x)$.
Second, $\textup{SL}_1(A)\cap Z(A) = \{x\in K^\times\,|\, x^n=1\}$ (see \textit{loc.~cit.}). 
\end{proof}

Let $Q=Q(a,b,K)$ be a quaternion algebra over a field $K$ of characteristic $\neq 2$, say with basis $\{1,u,v,uv\}$ and such that $u^2=a$ and $v^2=b$ are in $K^\times$ and $vu=-uv$. It follows that a given element $\alpha+\beta u+\gamma v+\delta uv$ commutes with another element $t+xu+yv+zuv$ in $Q$ if and only if $\gamma z=\delta y$, $\delta x =\beta z$ and $\gamma x=\beta y$.
Since the number of linearly independent solutions of this  system of equations on the variables $x$, $y$ and $z$ is $1$, except when $\beta$, $\gamma$ and $\delta$ are all $0$, we see that the Lie algebra $\mathfrak{sl}_1(Q)$ consisting of the 
pure quaternion elements of $Q$, with commutator bracket, 
is CA. 
This illustrates case (b2) of Theorem~\ref{t:CTLieAlgebras}.

Suppose now 
$K=\mathbb{Q}_p$, $b=p>2$, $a$ is an integer $>0$ and $<p$ that is not a square modulo $p$, and let $P$ be the maximal ideal of the maximal $\mathbb{Z}_p$-order $R$ of $Q$. 
The group $G=\textup{SL}_1(R)=\textup{SL}_1(Q)\cap R$ 
is 
not CA since $-1\in Z(G)$. 
Consider the subgroup $\textup{SL}_1^1(D)={\textup{SL}_1(Q)\cap (1+P)}$.
Since $1$ is the only element of it which has pure quaternion part $=0$, it follows from the previous paragraph that this group is CA.
 Therefore, the $p$-adic analytic pro-$p$ group $\textup{SL}_1^1(D) \ltimes R$ has Lie algebra of type (c) in Theorem \ref{t:CTLieAlgebras}.

 Let $D$ be a locally compact division algebra which is central over a $p$-adic field $K$. Denote by $R$ its ring of integers (i.e., the unique maximal compact subring of $D$; see \cite[p.~30]{Weil:67}) and let an element $\pi$ in $D$ be a generator of the maximal ideal of $R$. For $n \geq 1$, let
 \[ \textup{SL}_1^n(D) = \{ g \in SL_1(D) \mid g \equiv 1 \!\!\!\pmod {\pi^n R}\}\]
be the $n$-th congruence subgroup of $\textup{SL}_1(D)$.

It is well known, cf. \cite[Th.~7]{Riehm1970} and \cite[Prop.~4.3]{Ershov2022}, that $\textup{SL}_1^1(D)$ is a pro-$p$ group and  $\textup{SL}_1(D) \simeq \Delta \ltimes \textup{SL}_1^1(D)$, where $\Delta$ is the group of roots of unity in $W$ (the unramified extension of $K$ as in the proof of Proposition~\ref{LieDivisionAlg}) of order coprime to $p$ and reduced norm $1$. In particular, the first congruence subgroup is the maximal pro-$p$ subgroup of $\textup{SL}_1(D)$.

\begin{theorem} \label{t:sl_1^n}
 Let $p$ and $\ell$ be primes. Let $K$ be a finite extension of $\mb{Q}_p$ with ramification index $e$ and let $D$ be a central division algebra of degree $\ell$ over $K$. Then:
 \begin{enumerate}
  \item $\textup{SL}_1(D)$ is CA if and only if $K$ does not contain non-trivial $\ell$-th roots of unity. In this case, $\textup{SL}_1(D)$ is CSA.
  \item $\textup{SL}_1^1(D)$ is CA if and only if either $p=\ell$ and $K$ does not contain non-trivial $\ell$-th roots of unity or $p\neq \ell$.
  \item If $n \geq e\ell/(p-1)$, then $\textup{SL}_1^n(D)$ is CA.
 \end{enumerate}
\end{theorem}

\begin{proof}
Since $D$ has prime degree, the only possible elements of $\textup{SL}_1(D)$ with non-abelian centralizer are its central elements, i.e. the $\ell$-th roots of unity in $K$. This proves the first part of (1).

Now, since $\textup{SL}_1^1(D)$ is pro-$p$, the roots of unity that it possibly contains are of order $p^r$, for some $r$. Thus $\textup{SL}_1^1(D)$ contains some central element of $D$ if and only if $p=\ell$ and $K$ contains a $p$-th root of unity. Next, under the hypotheses of item (3), the congruence subgroup $\textup{SL}_1^n(D)$ is torsion-free by \cite[Prop.~4.3(c)]{Ershov2022}, so it does not contain non-trivial roots of unity. By the previous reasoning it has the CA property.

Finally we prove the assertion on CSA. Let $1 \neq g \in G=\Sl_1(D)$, $C=C_G(g)$ and $h \in G \smallsetminus C$. Let $F=C_D(g)$ and $\varphi:D\to D$ be the map $x\mapsto hxh^{-1}$. Then $F$ is a subfield of dimension $\ell$ over $K$, and so is  $\varphi(F)$.  The subextension $F \cap \varphi(F)$ is either $F$ or $K$.

In the first case, it follows that $\varphi(F)$ is equal to $F$ because both fields have dimension $\ell$ over $K$. So $\varphi:F\to F$ is an automorphism of the field extension $F|K$, and consequently the $\ell$-th power $h^\ell$ commutes with $g$. By the CA property, either $h \in C$, which we are assuming is not the case, or $h^\ell = 1$.  It follows that $h$ is an $\ell$-th root of $1$ in an extension of degree $\ell$ of $K$. So $\ell$ divides $\ell-1$, which is impossible.

In the second case, we find that $C \cap \varphi(C)$ is the set of elements of reduced norm 1 in $K$, which are $\ell$-th roots of $1$. So $C$ is malnormal.
\end{proof}

For instance, if $n \geq e\ell$, note that $\textup{SL}_1^n(D)$ is a uniform pro-$p$ CA group, since it is powerful by \cite[Prop.~4.3(b)]{Ershov2022}.
This contributes to \cite[Question~1]{ShZZ:19}.

On the other hand, assuming $K=\mathbb{Q}_p$, then $\textup{SL}_1^1(D)$ is CA, and $\textup{SL}_1(D)$ is CA if and only if $p$ is odd and $\ell$ does not divide $p-1$.

\begin{proof}[Proof of Theorem~\ref{t:CSA}] 
Applying Theorem~\ref{t:sl_1^n} for  $K=\Q_p$, $p$ odd, and $\ell$ prime coprime to $p-1$
yields $\Sl_1(D)/\Sl_1^1(D)\simeq \operatorname{Ker}(N\colon\F_{p^\ell}^\times\to\F_p^\times)$, and thus has order
$\tfrac{p^\ell-1}{p-1}$ which is coprime to $p-1$.
\end{proof}

\section{Some universal $p$-Frattini extensions}
\label{s:uniFrat}

\subsection{Group extensions with abelian kernel}
\label{ss:nonspl}
Let $G$ be a finite group, let $p$ be a prime number, and let $M$ be a finitely generated left $\Z_p[G]$-module. 
A short exact sequence of profinite groups
\begin{equation}
\label{eq:ext}
\bos:\qquad\xymatrix{\triv\ar[r]&M\ar[r]^{\iota}&E\ar[r]^{\pi}&G\ar[r]&\triv}
\end{equation}
is called an {\it extension of $G$} by the left $\Z_p[G]$-module $M$. Note that -- by definition --
$E$ is a virtually abelian, virtual pro-$p$ group
satisfying $\iota(M)\subseteq O_p(E)$.
If $\bft\colon\xymatrix{\triv\ar[r]&M'\ar[r]^{\iota^\prime}&E'\ar[r]^{\pi^\prime}&G\ar[r]&\triv}$
is another such extension of $G$,  a pair of homomorphisms $(\alpha,\beta)$, where $\alpha\colon M\to M'$
is a homomorphism of left $\Z_p[G]$-modules and $\beta\colon E\to E'$ is a group homomorphism is called a {\it morphism} from 
$\bos$ to $\bft$, provided the
diagram
\begin{equation}
\label{eq:exthom}
\xymatrix{\bos:&\triv\ar[r]&M\ar[d]^\alpha\ar[r]^{\iota}&E\ar[d]^\beta\ar[r]^{\pi}&G\ar@{=}[d]\ar[r]&\triv\\
\bft:&\triv\ar[r]&M'\ar[r]^{\iota^\prime}&E'\ar[r]^{\pi^\prime}&G\ar[r]&\triv}
\end{equation}
commutes. For short we denote the isomorphism class of the 
extension $\bos$ by $[\bos]$. It is well-known that the equivalence classes of group extensions of $G$ by the left 
$\Z_p$-module $M$ are in canonical one-to-one correspondence with elements in the group $H^2(G,M)$
(cf. \cite[Th.~1, p.71]{KG:coh}, \cite[Th. IV.3.12]{brown:coh}). In particular, there exists a morphism $(\alpha,\beta)$ from $\bos$
to $\bft$, if $[\bft]=H^2(\alpha)([\bos])$, where $H^2(\alpha)\colon H^2(G,M)\to H^2(G,M')$ is the map induced by $\alpha$.
Sometimes it will be convenient to consider the category $\big(\frac{\Z_p[G]}{}\big)$ of all extensions of the finite group $G$
by some finitely generated left $\Z_p[G]$-modules $M$.
 The extension
\eqref{eq:ext} is said to be a {\it $p$-Frattini extension}, if $\iota(M)\subseteq\Phi(E)$, 
where $\Phi(E)$ denotes the Frattini subgroup of the profinite group $E$ (i.e., 
the intersection of the maximal closed proper subgroups of $E$; see \cite[Sec.~22.1]{FJ:fiar}).
The following fact will turn out to be useful for our purpose.
\begin{fact}
\label{fact:fratspl}
Let 
$\bos:\xymatrix{\triv\ar[r]&M\ar[r]^{\iota}&E\ar[r]^{\pi}&G\ar[r]&\triv}$
be a Frattini extension of profinite groups.
Then $\bos$ is split, i.e., $[\bos]=0$, if, and only if,
 $M=\{1\}$.
\end{fact}

\begin{proof} 
If $M=\{1\}$, then $\bos$ is obviously split, and there is nothing to prove. So assume that $\bos$ is split, and
let $\eta\colon G\to E$ be a continuous group homomorphism satisfying $\pi\circ\eta=\iid_G$. 
The profinite group $E$ is generated by $\iota(M)$ and $\eta(G)$, with $\iota(M)\subseteq\Phi(E)$; 
so $E=\eta(G)$. 
For $m\in M$, if $\iota(m)=\eta(g)$ then $g=1$; thus $M=\{1\}$. 
\qedhere
\end{proof}

For further details the interested reader is referred to \cite[Ch.~5]{KG:coh}. 
It is well-known that the classical Krull-Schmidt theorem (cf. \cite[\S10.2, p.~227]{KG:coh}) holds for complete discrete valuation domains (\cite{dave:modrep}). 
Hence a version of a  W.~Gasch\"utz' theorem (\cite[Th.~10, p.~272]{KG:coh})
implies that the category $\big(\frac{\Z_p[G]}{}\big)$ has a unique minimal projective cover 
\begin{equation}
\label{eq:Uext}
\tbos_G:\qquad\xymatrix{\triv\ar[r]&\Omega_2(G,\Z_p)\ar[r]^-{\iota}&\tG\ar[r]^-{\pi_G}&G\ar[r]&\triv},
\end{equation}
where $\Omega_2(G,\Z_p)$ denotes the second Heller translate of $\Z_p$.
The extension class $[\tbos_G]\in H^2(G,\Omega_2(G,\Z_p))=\Ext^2_{\Z_p[G]}(\Z_p,\Omega_2(G,\Z_p))$
(cf. \cite[Ch.~IV, Th.~4.1]{mcl:hom})
corresponds to the exact sequence (cf. \cite[Sec.~III.1]{mcl:hom})
\begin{equation}
\label{eq:extcl}
\xymatrix{0\ar[r]&\Omega_2(G,\Z_p)\ar[r]&P_1\ar[r]^{\partial_1}&P_0\ar[r]^\eps&\Z_p\ar[r]&0,}
\end{equation} 
where $(P_\bullet,\partial_\bullet,\eps)$ is a minimal projective resolution of the trivial left $\Z_p[G]$-module $\Z_p$.
It is well-known that $\pi_G\colon\tG\to G$ coincides with the universal $p$-Frattini cover with abelian kernel (cf. \cite[\S22.11]{FJ:fiar}),
i.e., $\iota(\Omega_2(G,\Z_p))\subseteq O_p(\Phi(G,\Z_p))
\cap\Phi(\Phi(G,\Z_p))$.

\subsection{Properties of the universal $p$-Frattini cover with abelian kernel}
\label{ss:propFratt}
Let $G$ be a finite group and let $\pi_G\colon\Phi\to G$ be its universal $p$-Frattini cover with abelian kernel (cf. \ref{eq:Uext}) for some prime number~$p$. Since a finite group is of finite cohomological $p$-dimension if, and only if, its order is not divisible by $p$, Fact~\ref{fact:fratspl} implies the following:
\begin{fact}
\label{fact:triv}
For a finite group $G$, the epimorphism $\pi_G\colon\Phi\to G$ is injective,
if and only if, $|G|$ is a unit in $\Z_p$.
\end{fact}

Let $H$ be a subgroup of $G$.  The $H$-module $\rst^G_H(\Omega)$ formed by restriction of the $G$-action  can be decomposed as
\begin{equation}
\label{eq:Urst}
\rst^G_H(\Omega)\simeq \Omega_2(H,\Z_p)\oplus P
\end{equation}
for some finitely generated projective left $\Z_p[H]$-module $P$. Thus:

\begin{prop}
\label{prop:inher}
Let $H$ be a subgroup of a finite group $G$. 
Then $\pi_G^{-1}(H)$ contains an abelian normal subgroup $P$, which is
a projective left $\Z_p[H]$-module, such that $\pi_G^{-1}(H)/P\simeq\Phi(H,\mathbb{Z}_p)$,
and the induced map $\pi_G^{-1}(H)/P\to H$ coincides with the
universal $p$-Frattini cover with abelian kernel of $H$.
\end{prop}

The following property is an immediate consequence of  \eqref{eq:extcl}.

\begin{prop}
\label{prop:rk1}
Let $G$ be a finite $p$-group. Then
\begin{equation}
\label{eq:rk1}
\rk_{\Z_p}(\Omega)=(d_G-1)\cdot |G|+1, 
\end{equation}
where $\rk_{\Z_p}(M)$ is cardinality of a minimal generating subset of a finitely generated free ${\Z_p}$-module $M$ and $d_G=\dim_{\F_p}(G/\Phi(G))$ is the cardinality of a minimal generating subset of $G$.
\end{prop}

\subsection{The universal $p$-Frattini extension with elementary abelian kernel}
\label{ss:propFrattel}
Next to
the universal $p$-Frattini 
cover
$\pi_G\colon\Phi\to G$ with abelian kernel,
there exists also a
\emph{universal $p$-Frattini extension $\pi_{G,\F_p}\colon\eG\to G$ with elementary abelian kernel},
where $\operatorname{Ker}(\pi_{G,\F_p})\simeq \Omega_2(G,\F_p)$ and $\Omega_2(G,\F_p)$ fits into an exact sequence
\begin{equation}
\label{eq:extclmodp}
\xymatrix{0\ar[r]&\Omega_2(G,\F_p)\ar[r]&Q_1\ar[r]^{\partial_1}&Q_0\ar[r]^\eta&\F_p\ar[r]&0,}
\end{equation} 
which 
is a minimal projective resolution of the trivial left $\F_p[G]$-module $\F_p$.

\subsection{Central involutions}
\label{ss:ceninv}
From now on we assume that $p=2$ and that the finite group $G$ has a central involution $\sigma$, 
i.e., $\sigma^2=1\not=\sigma$ and $\sigma\in\Zen(G)$. 
For short we call a finitely generated left $\Z_2[G]$-module $M$ a
{\it $\Z_2[G]$-lattice}, if 
its additive group is torsion-free.
Every left $\Z_2[G]$-lattice $M$ admits the
$\Z_2[G]$-submodules
\begin{align}
M_+&=\{\,m\in M\mid \sigma\cdot m=m\,\}=M^{\langle\sigma\rangle},\label{eq:quat1}\\
\intertext{and}
M_-&=\{\,m\in M\mid\sigma\cdot m=-m\,\}\label{quat2}.
\end{align}
One easily verifies that $M_+$ and $M_-$ are saturated $\Z_2[G]$-submodules, i.e.,
 for $\sat_M(M_+)=\{\,m\in M\mid\exists r\in\Z_2\smallsetminus\{0\}: r\cdot m\in M_+\,\}$ one has
 $\sat_M(M_+)=M_+$ and $\sat_M(M_-)=M_-$. In particular, $M^-=M/M_+$ and $M^+=M/M_-$ are again
 $\Z_2[G]$-lattices.
 
\subsection{The $(-1)$-universal $2$-Frattini extension with abelian kernel}
\label{ss:propFrattone}
Let $G$ be a finite group with a central involution $\sigma$. Put
\begin{equation}
\label{eq:-uni1}
\Upsilon=\hGt=\Phi(G,\mathbb{Z}_2)/\iota(\Omega_2(G,\Z_2)_+),
\end{equation}
and let
\begin{equation}
\label{eq:-uni2}
\hat{\bos}_G:\qquad\xymatrix{\triv\ar[r]&\Omega_2(G,\Z_2)^-\ar[r]^-{j}&\hGt\ar[r]^-{\tau_G}&G\ar[r]&\triv}
\end{equation}
be the extension of $G$ induced by \eqref{eq:-uni1}. In particular, one has $\Omega^-=\Omega_2(G,\Z_2)^-=\Omega_2(G,\Z_2)/\Omega_2(G,\Z_2)_+$. If no confusion arises we will from now on assume that 
$j\colon\Omega^-\to\Upsilon$
is given by inclusion.
The projectivity of \eqref{eq:Uext} 
provides our next result.

\begin{prop}
\label{prop:-uni}
Let $G$ be a finite group and let $\sigma$ be a central involution of $G$.
The extension $\hbos_G$ of \textup{(\ref{eq:-uni2})} has the following properties:
\begin{itemize}
\item[(a)] It is universal for extensions of $G$
by  $\Z_2[G]$-lattices $M$ on which $\sigma$ acts by $-1$, i.e., if 
\begin{equation}
\label{eq:-uni3}
\xymatrix{\bft\colon&\triv\ar[r]&M\ar[r]^\iota&X\ar[r]^\rho&G\ar[r]&\triv}
\end{equation}
is an extension of $G$ by the $\Z_2[G]$-lattice $M$ satisfying $M=M_-$,
then there exists a homomorphism of extensions $(\alpha,\beta)\colon\hbos_G\to\bft$.
\item[(b)] $\hat{\bos}_G$ is a $2$-Frattini extension, i.e., 
\begin{equation}
\image(j)\subseteq\Phi(\Upsilon)\cap O_2(\Upsilon).
\end{equation}
\end{itemize}
\end{prop}

\begin{proof}
{\rm (a)} As \eqref{eq:Uext} is projective, there exists a homomorphism of extensions
$(\alpha_\circ,\beta_\circ)\colon\tbos\to\bft$. Moreover, as $M=M_-$ one has that
$\alpha_\circ(\Omega_2(G,\Z_2)_+)=0$ and $\beta_\circ(\iota(\Omega_2(G,\Z_2)_+)=0$.
Thus $(\alpha_\circ,\beta_\circ)$ induces a homomorphism of extensions
$(\alpha,\beta)\colon \hbos\to\bft$ as claimed.

\noindent
{\rm (b)} As \eqref{eq:Uext} is a $2$-Frattini extension, one has
$\iota(\Omega_2(G,\Z_2))\subseteq\Phi(\Phi(G,\Z_2))$. 
Hence, $j(\Omega^-)\subseteq\Phi(\Upsilon)\cap O_2(\Upsilon)$. 
\end{proof}

The following proposition describes some features of the $2$-Frattini extension $\hbos_G$ for a finite group $G$. 

\begin{prop}
\label{prop:2grp}
Let $G$ be a finite group with a central involution $\sigma$. 
Then the following holds.
\begin{itemize}
\item[(a)] If $G$ is a finite $2$-group, then 
$\rk_{\Z_2}(\Omega^-)=(d_G-1)\cdot\frac{|G|}{2}$.
\item[(b)] The epimorphism $\tau_G\colon\Upsilon\to G$ is an isomorphism if, and only if,  $P\in\Syl_2(G)$ is cyclic and $G\simeq P\cdot O_{2^\prime}(G)$. Here
$\Syl_2(G)$ denotes the set of all $2$-Sylow subgroups of $G$.
\item[(c)] 
Let $H$ be a subgroup of $G$. Then $j(\Omega^-)\subseteq \tau_G^{-1}(H)$.
Moreover, if $H$ contains $\sigma$ then there exists a $\Z_2[H]$-submodule $N$
of $\rst^G_H(\Omega^-)$ such that
\begin{equation}
\tau_G^{-1}(H)/j(N)\simeq\Upsilon(H,\sigma,2). 
\end{equation}
Furthermore, there exists
a subgroup $Y$ of $\tau_G^{-1}(H)$ isomorphic to $\Upsilon(H,\sigma,2)$ such that $\tau_G^{-1}(H)= Y\cdot j(N)$.
\end{itemize}
\end{prop}

\begin{proof} 
Put $Z=\langle\sigma\rangle$.

 \noindent
 {\rm (a)}
As $\rk_{\Z_2}(\Omega_2(G,\Z_2))=(d_G-1)|G|+1$ and $\Omega_2(Z,\Z_2)\simeq\Z_2$, one concludes from \eqref{eq:Urst} and \eqref{eq:rk1} that
\begin{equation}
\label{eq:rkmin1}
\rst^G_Z(\Omega)\simeq \Omega_2(Z,\Z_2)\oplus m\cdot\Z_2[Z]
\end{equation}
for $m=(d_G-1)\cdot\frac{|G|}{2}$. This yields (a).

\noindent
{\rm (b)} Note that $\tau_G$ is an isomorphism if, and only if $\Omega_2(G,\Z_2)^-=0$.
Let $P\in\Syl_2(G)$. Then $Z\subseteq P$ and, as in (a), one has 
\begin{equation}
\label{eq:res2Syl}
\rst^G_P(\Omega^-)\simeq\Omega_2(P,\Z_2)^-\oplus n\cdot\Z_2[P/Z]
\end{equation}
for some integer $n\geqslant 0$. Thus $\Omega_2(G,\Z_2)^-=0$ implies
$n=0$ and $\Omega_2(P,\Z_2)^-=0$. From $n=0$ one concludes
that $\rst^G_P(\Omega_2(G,\Z_2))\simeq\Omega_2(P,\Z_2)$. Moreover, $\Omega_2(P,\Z_2)^-=0$ implies $d_P=1$ by item (a),
i.e., $P$ is cyclic.
As $P$ is cyclic, one has that $\Omega_2(P,\Z_2)\simeq\Z_2$ is the trivial left $\Z_2[P]$-module isomorphic to $\Z_2$
as an abelian group.
Let $\psi\colon G\to\Z_2^\times$ denote the representation associated to the left $\Z_2[G]$-module $\Omega_2(G,\Z_2)$.
Since $\Z_2^\times\simeq\Z_2\times\{\pm1\}$, and as $G$ is a finite group, one has that $\image(\psi)\subseteq \{\pm1\}$. 
Suppose $-1\in\image(\psi)$. Then there exists $g\in G$ such that $\psi(g)=-1$. Let $g_+$ and $g_-$ be the elements in
$\langle g\rangle$ satisfying $g=g_+\cdot g_-$, where $g_+$ is of $2$-power order and $g_-$ is of odd order.
Then, by construction, $\psi(g_-)=1$, and $\psi(g_+)=1$ as $g_+$ is contained in some Sylow $2$-subgroup of $G$, and
$\rst^G_P(\Omega_2(G,\Z_2))\simeq\Omega_2(P,\Z_2)$ is a trivial left $\Z_2[P]$-module, a contradiction,
showing that $\Omega_2(G,\Z_2)\simeq\Z_2$ is the trivial left $\Z_2[G]$-module. In particular,
$G$ has periodic mod $2$ cohomology of period $2$.

As $(G:P)$ is odd, the restriction map $H^2(G,\Z_2)\to H^2(P,\Z_2)$ is injective (because $\text{cor} \circ \text{res} = (G:P)$). Thus, by Tate duality
(cf. \cite[Prop.~VI.7.1]{brown:coh}), the map
$P\simeq H_1(P,\Z_2)\to H_1(G,\Z_2)=G/[G,G]O^2(G)\simeq \Z_2/|G|\cdot\Z_2$ is surjective and thus an isomorphism showing that $P[G,G]O^2(G)=G$ and $P\cap [G,G]O^2(G)=\{1\}$.
Hence $[G,G]O^2(G)$ is of odd order and thus coincides with $O_{2^\prime}(G)$. Otherwise, if $G=P\cdot O_{2^\prime}(G)$
for $P$ a cyclic $2$-group, then $G$ has periodic mod $2$
cohomology, i.e., $\Omega_2(G,\Z_2)\simeq\Z_2$ and consequently $\Omega_2(G,\Z_2)^-=0$, i.e., $\tau_G\colon\Upsilon\to G$ is an isomorphism.

\noindent
{\rm (c)} By \eqref{eq:Urst}, one has
\begin{equation}
\label{eq:rkmin2}
\rst^G_H(\Omega)\simeq \Omega_2(H,\Z_2)\oplus Q
\end{equation}
for some finitely generated projective left $\Z_2[H]$-module $Q$, and hence
\begin{equation}
\label{eq:rkmin}
\rst^G_H(\Omega^-)\simeq \Omega_2(H,\Z_2)^-\oplus Q^-.
\end{equation}
Hence $N=j(Q^-)\subseteq j(\Omega^-)$
has the desired property.
Proposition~\ref{prop:-uni}(a) shows that one has a homomorphism of extensions
\begin{equation}
\label{eq:Xuni1}
\xymatrix{\triv\ar[r]&\Omega_2(H,\Z_2)^-\ar[r]^-{j_H}\ar[d]_\alpha&\Upsilon(H,\sigma,2)\ar[r]^-{\tau_H}\ar[d]_\beta&
H\ar@{=}[d]\ar[r]&\triv\\
\triv\ar[r]&\Omega_2^-\ar[r]^-j & \tau_G^{-1}(H)\ar[r]^{\tau_G}&H\ar[r]&\triv}
\end{equation}
and, by the previous remark $\alpha$ is injective. Hence $\beta$ is injective, and putting $Y=\image(\beta)$ one concludes that $\tau_G^{-1}(H)=Y\cdot j(N)$.
\end{proof}
From Proposition~\ref{prop:2grp} one concludes the following property.

\begin{lemma}
\label{lem:finsub}
Let $G$ be a finite group with a unique involution $\sigma$ and 
let $x$ be an element in $\Upsilon$ such that $\tau_G(x)$ has even order.
Then the centralizer $C_{\Upsilon}(x)$ is a finite group and
\begin{equation}
\label{eq:Cg1}
C_{\Upsilon}(x)\simeq P\ltimes O ,
\end{equation}
where $P\in\Syl_2(C_{\Upsilon}(x))$ is cyclic and $O\in\Odd(G)$.
\end{lemma}

\begin{proof}
Let $C=C_{\Upsilon}(x)$ and $\bar{x}=\tau_G(x)$. There exists $d$ in $\N$ such that $\bar{x}^d=\sigma$, whence
\begin{equation}
\label{eq:intOm}
C\cap \Omega^-= C_{\Omega^-}(x)=C_{\Omega^{-}}(\bar{x})\subseteq C_{\Omega^-}(\bar{x}^d)
= C_{\Omega^-}(\sigma) =\{1\}.
\end{equation}
So $C$ is isomorphic to a finite subgroup $H$ of $G$ containing $\bar{x}$ and thus $\sigma$. Hence with the notation of Proposition~\ref{prop:2grp}(c) the canonical map $\rho\colon C\to \Upsilon/j(N)$, $\rho(c)=cj(N)$, yields a section for $\tau_H$, i.e.,  the $2$-Frattini
extension 
\begin{equation}
\label{eq:X}
\hbos_X\colon\ \xymatrix{\triv\ar[r]&\Omega_2(H,\mathbb{Z}_2)^-\ar[r]^-{j}&\Upsilon(H,\sigma,2)/j(N)\ar[r]^-{\tau_H}&
H\ar[r]&\triv}
\end{equation}
splits. Therefore, by Fact~\ref{fact:fratspl}, one has that $\tau_H$ is an isomorphism, and hence Proposition~\ref{prop:2grp}(b) yields the claim.
\end{proof}

\subsection{Ranks}
\label{ss:rk}
Let $(G,\sigma)$ be a finite group with a central involution $\sigma$.  In general there is no formula for the dimension of $\Omega_2(G,\F_2)$, but it has to be calculated from the $2$-modular representation theory of $G$. Nevertheless, the rank of $\Omega=\Omega_2(G,\Z_2)$ and $\Omega^-=\Omega_2(G,\Z_2)^-$ are related
in a straightforward way which will be described next.

\begin{prop}
\label{prop:rk2}
Let $G$ be a finite group with a central involution $\sigma$. Put $Z=\langle\sigma\rangle$ and
$\bG=G/Z$. Then
\begin{itemize}
\item[(a)] $\rk_{\Z_2}(\Omega)=\dim_{\F_2}(\Omega_2(G,\F_2))$.
\item[(b)] $\rk_{\Z_2}(\Omega^-)=\tfrac{1}{2}(\rk_{\Z_2}(\Omega)-1)$.
\item[(c)] If $Z\subseteq\Phi(G)$, then 
\begin{equation}
\label{eq:M2}
\dim_{\F_2}(\Omega_2(\bG,\F_2))-1=\tfrac{1}{2}(\dim_{\F_2}(\Omega_2(G,\F_2))-1).
\end{equation}
\item[(d)] If $Z\subseteq\Phi(G)$, then there exists an $\F_2[\bG]$-submodule $\Omega_2(\bG,\F_2)^\circ$ of $\Omega_2(\bG,\F_2)$
of codimension $1$ satisfying $\Phi(\bG,\F_2)/\Omega_2(\bG,\F_2)^\circ\simeq G$ and an
isomorphism 
\begin{equation}
\label{eq:eqcrit0}
\gamma\colon\Omega_2(\bG,\F_2)^\circ\longrightarrow \Omega^-/2\Omega^-.
\end{equation}
\end{itemize}
\end{prop}
\begin{proof}
(a) If $(P_\bullet,\partial_\bullet,\eps)$ is a minimal projective resolution of the trivial left $\Z_2[G]$-module $\Z_2$,
 $(\F_2\otimes_{\Z_2} P_\bullet,\partial_\bullet^\otimes,\eps^\otimes)$ is a minimal projective resolution of the trivial left 
 $\F_2[G]$-module $\F_2$. As all $\Z_2[G]$-modules $P_k$ are $\Z_2[G]$-lattices, this yields the claim.
 
 \noindent
 (b) Let $d=\tfrac{1}{2}(\rk_{\Z_2}(\Omega) -1)$. Then, as $\Omega_2(Z,\Z_2)\simeq\Z_2$ is the
 trivial left $\Z_2[Z]$-module, one has $\rst^G_Z(\Omega)=\Z_2\oplus\Z_2[Z]^d$, which 
 yields the claim.
 
 \noindent
 (c) If $\tau_S\colon Q_S\to S$ is the minimal projective cover of the irreducible left $\F_2[\bG]$-module $S$ in the category of left $\F_2[G]$-modules,
 $\tau^\otimes_S\colon  \F_2[\bG]\otimes_G Q_S\to S$ is the minimal projective cover in the category of left
 $\F_2[\bG]$-modules. 
 As $Z\subseteq \Phi(G)$, $\inf_{\bG}^G\colon H^1(\bG,S)\to H^1(G,S)$
 are isomorphisms for all irreducible left $\F_2[\bG]$-modules $S$ (cf. \cite[Theorem~3.9]{wwil:splt},
 \cite[Theorem]{stamm:splt}). 
 Thus - using the notation of \eqref{eq:extclmodp}
 (cf \S\ref{ss:propFrattel}) - one has a short exact sequence
 \begin{equation}
 \label{eq:crit1}
 \xymatrix@C=0.5cm{
 0\ar[r]&\Omega_2(\bG,\F_2)\ar[r]&\F_2[\bG]\otimes_G Q_1\ar[r]^{\partial_1^\otimes}&
 \F_2[\bG]\otimes_G Q_0\ar[r]&\F_2\ar[r]&0
 }.
 \end{equation}
Since 
 \begin{equation}
 \label{eq:eqcrit3}
 \dim_{\F_2}(\F_2[\bG]\otimes_G P)=\tfrac{1}{2}\cdot\dim_{\F_2}(P)
\end{equation}
for any finitely generated projective left $\F_2[G]$-module $P$, \eqref{eq:extcl} yields the claim.

\noindent
(d) As $\Xi=Z.\Omega^-/2\Omega^-\subseteq\Upsilon/2\Omega^-$
is a normal subgroup of exponent $2$, it is an elementary abelian normal $2$-group contained in 
$\Phi(\Upsilon/2\Omega^-)\cap O_2(\Upsilon/2\Omega^-)$.
Thus, by universality, there exist continuous homomorphisms 
$\alpha\colon\Omega_2(\bG,\F_2)\to\Xi$ and
$\beta\colon\Phi(\bG,\F_2)\to\Upsilon/2\Omega^-$
making the diagram
\begin{equation}
\label{eq:dia1}
\xymatrix{\triv\ar[r]&\Omega_2(\bG,\F_2)\ar[r]^-{\iota_{/2}}\ar[d]_\alpha&\Phi(\bG,\F_2)\ar[r]^-{\pi_{\bG,/2}}\ar[d]_\beta&
\bG\ar@{=}[d]\ar[r]&\triv\\
\triv\ar[r]&\Xi\ar[r]^-{\eta} &\Upsilon/2\Omega^- \ar[r]^-{\rho_G}&\bG\ar[r]&\triv}
\end{equation}
commute. As $\Xi\subseteq\Phi(\Upsilon/2\Omega^-)$, $\beta$ is surjective, and thus, by the snake lemma, $\alpha$ is surjective as well.
Let $\tau^\circ_G\colon\Upsilon/2\Omega^-\to G$ denote the canonical projection induced by $\tau_G$. Then one has a commutative diagram
\begin{equation}
\label{eq:dia2}
\xymatrix{\triv\ar[r]&\operatorname{Ker}(\tau_G^\circ\circ\beta)\ar[r]^-{\iota_{/2}^\prime}\ar[d]_\gamma&\Phi(\bG,\F_2)\ar[r]^-{\tau_G^\circ\circ\beta}\ar[d]_\beta&
G\ar@{=}[d]\ar[r]&\triv\\
\triv\ar[r]&\Xi^\circ\ar[r]^-{\eta} &\Upsilon(G,\sigma,2)/2\Omega_2(G,\Z_2)^- \ar[r]^-{\tau_G^\circ}&G\ar[r]&\triv},
\end{equation}
where $\Xi^\circ=\Omega_2(G,\Z_2)^-/2\Omega_2(G,\Z_2)^-$.
Thus $\Omega_2(\bG,\Z_2)^\circ=\operatorname{Ker}(\tau_G^\circ\circ\beta)$ is a submodule of $\Omega_2(\bG,\F_2)$ of codimension $1$ with the desired property. By the snake lemma, $\gamma$ is surjective. Since
$\dim_{\F_2}(\Omega_2(\bG,\Z_2)^\circ)=\dim_{\F_2}(\Xi^\circ)$ by (b) and (c), $\gamma$ must be an isomorphism.
\end{proof}


\subsection{$\Sl_2$-like finite groups with the $\Omega^\circ$-fixed point property}\label{ss:FixedPointProperty}
Under the notation of Proposition~\ref{prop:rk2}(d) we say that a finite group
$G$ with central involution $\sigma$ has the \emph{$\Omega_2(\bG,\F_2)^\circ$-fixed point property}, where $\bG=G/\langle\sigma\rangle$, if $(\Omega_2(\bG,\F_2)^\circ)^H=\{0\}$ for all $H\in\Odd(G)$.

\begin{fact}
\label{fact:Ofix}
Let $(G,\sigma)$ be a finite group with central involution $\sigma$ in $G$ with the
$\Omega_2(\bG,\F_2)^\circ$-fixed point property. Then for every non-trivial odd order
subgroup $H$ of $G$, one has $(\Omega^-)^H=\{0\}$.
\end{fact}

\begin{proof}
By \eqref{eq:eqcrit0} one has an isomorphism
\begin{equation}
\label{eq:eqcrit1}
\Omega^-/2\Omega^-\simeq\Omega_2(\bG,\F_2)^\circ.
\end{equation}
Suppose that there exists a non-trivial subgroup $H\in\Odd(G)$ satisfying $(\Omega^-)^H\not=0$.
Then, as $(\Omega^-)^H$ is a pure subgroup of $\Omega^-$, one has 
\begin{equation}
\label{eq:pure1}
\dim_{\F_2}(((\Omega^-)^H+2\Omega^-)/2\Omega^-)=\rk_{\Z_2}((\Omega^-)^H).
\end{equation}
and thus, by \eqref{eq:eqcrit1}, $(\Omega_2(\bG,\F_2)^\circ)^H\not=0$, a contradiction. Hence, for all
$H\in\Odd(G)$, $H\not=\{1\}$ one has $(\Omega^-)^H=0$.
\end{proof}

\subsection{$\Sl_2$-like finite groups}
\label{s:Sl2}
Thanks to Lemma~\ref{lem:finsub}, we are now able to prove Theorem~\ref{thm:exA}.

\begin{proof}[Proof of Theorem \ref{thm:exA}]
Let $G$ be a finite $\Sl_2$-like group and let $x$ in $\Upsilon$ be such that $\tau_G(x)$ has even order.
By Lemma~\ref{lem:finsub}, $C_{\Upsilon}(x)$ is a finite group isomorphic to $P\ltimes H$, where $P$ is cyclic and $H=O_{2^\prime}(C_\Upsilon(x))$.
\end{proof}

\subsection{The generalized quarternion groups}
\label{ss:genquat}
For $n\geq 3$ let
\begin{equation}
Q_{2^n}=\langle\,x,y\mid x^{2^{n-1}}=y^4=1,\, yxy^{-1}=x^{-1},\, x^{2^{n-2}}=y^2\,\rangle
\end{equation}
denote the generalized quarternion group of order $2^n$. 
The following property is well-known (cf. \cite[p.~191]{Gorenstein}).

\begin{fact}
\label{fact:genquat}
The element $\sigma=y^2=x^{2^{n-2}}$  is the unique element in $Q_{2^n}$ of order $2$.
In particular, $\Zen(Q_{2^n})=\langle\sigma\rangle$.
\end{fact}

From Theorem~\ref{thm:exA} one deduces the following result.

\begin{cor}
\label{cor:GQct}
For $n\geq 3$ and $Q_{2^n}$ the generalized quaternion group, the pro-$2$ group $\Upsilon({Q}_{2^n},\sigma,2)$ is CA and virtually abelian.
\end{cor}

\begin{proof}
As $\Odd(Q_{2^n})=\{\{1\}\}$ and since $\sigma$ is the only involution in $Q_{2^n}$, $(Q_{2^n},\sigma)$ is $\Sl_2$-like.
Therefore, Theorem~\ref{thm:exA} yields the claim.
\end{proof}

As $\Omega_2(Q_{2^n},\Z_2)^-$ is the maximal (abelian open normal) subgroup of $\Upsilon(Q_{2^n},\sigma,2)$, Corollary~\ref{cor:GQct} answers \cite[Question~4]{ShZZ:19} affirmatively. 

\subsection{2-local $\Sl_2$-like groups}
\label{ss:2loc}
Another consequence of Theorem~\ref{thm:exA} is the following.

\begin{cor}
    \label{cor:New}
    Let $G$ be a $2$-local $\Sl_2$-like group
    with the $\Omega_2(\bG,\F_2)^\circ$-fixed point property.
    Then
    $\Upsilon(G,\sigma,2)$ is CA.
\end{cor}

Here we called a finite group $G$ to be $2$-local, if
$|\Syl_2(G)|=1$, i.e., $P_2\in\Syl_2(G)$ is normal in $G$.

\begin{proof}
Let $x\in\Upsilon=\Upsilon(G,2,\sigma)$ and $C=C_\Upsilon(x)$.

\noindent
{\bf Case 1:} $\ord(\tau_G(x))$ is even.
In this case one concludes from
Theorem~\ref{thm:exA} that 
$C=P\ltimes O_{2^\prime}(C)$,
where $P\in\Syl_2(C)$ is cyclic.
As $G$ is $2$-local, $P_2\in\Syl_2(G)$ is normal in $G$ and
$P=P_2\cap C$ is normal in $C$.
As $P\cap O_{2^\prime}(C)=\{1\}$, one concludes that
$C\simeq P\times O_{2^\prime}(C)$.
In particular, as $G$ is $\Sl_2$-like, $C$ is abelian.

\noindent
{\bf Case 2:} $\ord(\tau_G(x))$ is odd different to $1$:
As $\langle\tau_G(x)\rangle$ is of odd order and non-trivial, 
Fact~\ref{fact:Ofix} implies 
\begin{equation}
\label{eq:TA1}
C\cap\Omega^-=(\Omega^-)^{\langle \tau_G(x) \rangle}=\triv.
\end{equation}
In particular, $C$ is a finite subgroup of $\Upsilon$.
If its order is odd, it is abelian by condition (i) of the definition of being $\Sl_2$-like, and there is nothing 
to be proved. If $C$ is of even order, then $H=\tau_G(C)$ contains $\sigma$, and by the same argument which was used in the
proof of Lemma~\ref{lem:finsub} one concludes that the sequence
\begin{equation}
\label{eq:TA2}
\xymatrix{\triv\ar[r]&\Omega_2(H,\Z_2)^-\ar[r]^j&\Upsilon(H,\sigma,2)\ar[r]^-{\tau_{H}}&H\ar[r]&\triv}
\end{equation}
must split, implying that $\Omega_2(H,\Z_2)^-=0$.
Thus Proposition~\ref{prop:2grp} and condition (i) of the definition of being $\Sl_2$-like
imply that $C$ is abelian.

\noindent
{\bf Case 3:} $\tau_G(x)=1$: 
 In this case $x\in\Omega^-\smallsetminus\{1\}$.
As $\Omega^-$ is abelian, one has $\Omega^-\subseteq C$.
Suppose there exists $z$ in $C\smallsetminus \Omega^-$, and let $\bar{z}=\tau_G(z)$.
Then $\bar{z}\not=1_{G}$. 
If $\bar{z}$ would be of even order, there would exist $d\in\N$ satisfying $\sigma=\bar{z}^d$,
and thus $x=x^{-1}$, a contradiction. 
If $\bar{z}$ would be of odd order, then 
$x\in C\cap\Omega^-=(\Omega^-)^{\langle \bar{z}\rangle}$
contradicting Fact~\ref{fact:Ofix}. This shows that $C=\Omega^-$.
\end{proof}

\subsection{The $2$-modular representation theory of $\operatorname{PSL}_2(3)$}
\label{ss:2mod}
Note that $G=\Sl_2(3)$ has a unique $2$-Sylow subgroup isomorphic to the quaternion group of order $8$, thus containing a unique involution $\sigma$, and that the only non-trivial subgroups of odd order of $G$ have order $3$. 
In particular, $G$ is $\Sl_2$-like
and $\bG=G/\langle\sigma\rangle\simeq C_3\ltimes V_4$, where
$V_4$ denotes the Klein four-group.

As $\bG$ is $2$-local,  the irreducible
$\bF_2[\bG]$-modules $\{\bF_2,S_1,S_2\}$ correspond to the irreducible $\bF_2[C_3]$-modules and thus are all linear.
Let $P_0$ denote the projective indecomposable $\bF_2[\bG]$-module
such that $P_0/\operatorname{Rad}(P_0)\simeq \bF_2$. Then $P_0\simeq\idn_{C_3}^{\bG}(\bF_2)$, where $\bF_2$ denotes the trivial left $\bF_2[\bG]$-module. In particular, $\dim_{\bF_2}(P_0)=4$. Considering $V_4$ as left $\F_2[C_3]$-module, where
$C_3$ is acting by conjugation, yields that $\bF_2\otimes_{\F_2}V_4\simeq S_1\oplus S_2$. Thus, by the Willems-Stammbach theorem (cf. \cite{stamm:splt}, \cite{wwil:splt}), one concludes that 
$JP_0/J^2P_0\simeq S_1\oplus S_2$, where $J=\operatorname{Rad}(\bF_2[\bG])$. Hence the Loewy structure of $P_0$ can be described by
\begin{equation}
\label{eq:P0}
P_0=\begin{bmatrix}
    \bF_2\\
    S_1\oplus S_2\\
    \bF_2
\end{bmatrix}
\end{equation}
Defining $P_1$ as the minimal projective cover of $S_1$ and $P_2$ as
the minimal projective cover of $S_2$, one obtains
$P_1\simeq S_1\otimes P_0$ and $P_2\simeq S_2\otimes P_0$ with Loewy structure
\begin{equation}
\label{eq:Px}
P_1=\begin{bmatrix}
    S_1\\
    S_2\oplus \bF_2\\
    S_1
\end{bmatrix},
\qquad\qquad
P_2=\begin{bmatrix}
    S_2\\
    \bF_2\oplus S_1\\
    S_2
\end{bmatrix}.
\end{equation}
From these facts one concludes that the trivial left
$\bF_2[\bG]$-module $\bF_2$ has a minimal projective resolution
\begin{equation}
    \label{eq:minC3V4}
    \xymatrix{
    0\ar[r]&\Omega_2(\bG,\bF_2)\ar[r]&
    P_1\oplus P_2\ar[r]&P_0\ar[r]&\bF_2\ar[r]&0}
\end{equation}
and therefore, 
\begin{equation}
\label{eq:O2C3V4}
\Omega_2(\bG,\bF_2)=\begin{bmatrix}
    S_1\oplus S_2\oplus\bF_2\\
    S_1\oplus S_2
\end{bmatrix}
\end{equation}
In particular, the image of 
$\bF_2\otimes\Omega_2(\bG,\F_2)^\circ$ in
the Grothendieck group $G_0(\bF_2[\bG])$ coincides with $2[S_1]+2[S_2]$,
and thus $G$ has the $\Omega_2(\bG,\F_2)^\circ$-fixed point property.

\section*{Funding}
\label{s:funding}
\addcontentsline{toc}{section}{Funding}
The first author was partially supported by FAPEMIG [APQ-02750-24]. 
The second author gratefully acknowledges financial support by the
PRIN2022 “Group theory and its applications”. The third named author is grateful for financial support from FAPDF (Edital 09/2023) and FEMAT (Edital 01/2020).

\section*{Acknowledgements}
\label{s:acknowl}
\addcontentsline{toc}{section}{Acknowledgements}
The first author is thankful to B. Klopsch and I. Snopce for helpful conversations about this article.  
The second author is a member of the
Gruppo Nazionale per le Strutture Algebriche, Geometriche e le loro Applicazioni
(GNSAGA), which is part of the Istituto Nazionale di Alta Matematica (INdAM).

%

\end{document}